\documentclass{amsart}%
\usepackage{amsfonts}
\usepackage{amsmath}
\usepackage{amssymb}
\usepackage[unicode=true]{hyperref}
\usepackage{graphicx}%
\providecommand{\U}[1]{\protect\rule{.1in}{.1in}}
\newtheorem{theorem}{Theorem}
\theoremstyle{plain}

\newtheorem{corollary}{Corollary}

\newtheorem{definition}{Definition}
\newtheorem{example}{Example}

\newtheorem{remark}{Remark}

\numberwithin{equation}{section}
\begin{document}
\title[Garsia-Rodemich Spaces]{Garsia-Rodemich Spaces: Bourgain-Brezis-Mironescu space, embeddings and
rearrangement invariant spaces}
\author{Mario Milman}
\address{Instituto Argentino de Matematica}
\email{mario.milman@gmail.com}
\urladdr{https://sites.google.com/site/mariomilman/}
\thanks{The author was partially supported by a grant from the Simons Foundation (%
%TCIMACRO{\TEXTsymbol{\backslash}}%
%BeginExpansion
$\backslash$%
%EndExpansion
\#207929 to Mario Milman)}

\begin{abstract}
We extend the construction of Garsia-Rodemich spaces in different directions.
We show that the new space \textbf{B,} introduced by Bourgain-Brezis-Mironescu
\cite{bbm}, can be described via a suitable scaling of the Garsia-Rodemich
norms. As an application we give a new proof of the embeddings $BMO\subset$
\textbf{B }$\subset$ $L(n^{\prime},\infty).$ We then generalize the
Garsia-Rodemich construction and introduce the $GaRo_{X}$ spaces associated
with a rearrangement invariant space $X,$ in such a way that $GaRo_{X}=X,$ for
a large class of rearrangement invariant spaces. The underlying inequality for
this new characterization of rearrangement invariant spaces is an extension of
the rearrangement inequalities of \cite{milbmo}. We introduce Gagliardo
seminorms adapted to rearrangement invariant spaces and use our generalized
Garsia-Rodemich construction to prove Fractional Sobolev inequalities in this context.

\end{abstract}
\maketitle

\section{Introduction}

In their celebrated paper \cite{jn}, John-Nirenberg introduced the
space\footnote{For definiteness, from now on $Q_{0}$ will denote the unit cube
$(0,1)^{n}.$} $BMO(Q_{0}),$ and established the exponential integrability of
functions in $BMO(Q_{0}).$ To complement their result on $BMO$ functions,
John-Nirenberg introduced the $JN_{p}(Q_{0})$ spaces$,$which provide a scale
of conditions on the oscillation of functions: For $1<p<\infty,$ let%
\[
JN_{p}(Q_{0}):=JN_{p}=\{f\in L^{1}(Q_{0}):\left\Vert f\right\Vert _{JN_{p}%
}<\infty\},
\]
where\footnote{In this paper we assume that all subcubes have sides parallel
to the coordinate axes and we let $f_{Q}=\frac{1}{\left\vert Q\right\vert
}\int_{Q}f$ .}%
\begin{equation}
\left\Vert f\right\Vert _{JN_{p}}=\sup_{\{Q_{i}\}_{i\in I}\in P}\left\Vert
f\right\Vert _{JN_{p}}=\sup_{\{Q_{i}\}_{i\in I}\in P}\left\{
%TCIMACRO{\dsum \limits_{i\in I}}%
%BeginExpansion
{\displaystyle\sum\limits_{i\in I}}
%EndExpansion
\left(  \left\vert Q_{i}\right\vert ^{\frac{1}{p}-1}\int_{Q_{i}}\left\vert
f(x)-f_{Q_{i}}\right\vert dx\right)  ^{p}\right\}  ^{1/p}, \label{definida}%
\end{equation}

\[
P=\{\{Q_{i}\}_{i\in I}:\{Q_{i}\}_{i\in I}\text{ countable subcubes of }%
Q_{0}\text{ with pairwise disjoint interiors}\}.
\]
Then, we see that (cf. \cite[pag 423]{jn})
\[
\lim_{p\rightarrow\infty}\left\Vert f\right\Vert _{JN_{p}}=\left\Vert
f\right\Vert _{BMO}.
\]
John-Nirenberg \cite{jn} then proceeded to obtain the corresponding
intermediate integrability results for $JN_{p}$ functions, which we formulate
here as embeddings: For $1<p<\infty,$ we have%
\begin{align}
JN_{p}  &  \subset L(p,\infty),\label{delaintro3}\\
\text{ \ \ \ \ }\left\Vert f-f_{Q_{0}}\right\Vert _{L(p,\infty)}  &  \leq
c_{p}\left\Vert f\right\Vert _{JN_{p}},\nonumber
\end{align}
where $L(p,\infty)$ denotes the Marcinkiewicz weak type $L^{p}$
space\footnote{defined by the condition%
\[
\left\Vert f\right\Vert _{L(p,\infty)}=\sup_{t}t\left\vert \{\left\vert
f\right\vert >t\}\right\vert ^{1/p}<\infty.
\]
}, and $c_{p}$ is an absolute constant that does not depend on $f$.

We note that, as $p\rightarrow\infty,$ the correct limiting Marcinkiewicz
condition is the exponential class and the resulting limiting inequality is
one of the possible formulations of the celebrated John-Nirenberg
Lemma\footnote{The natural condition in our context is via the
Bennett-DeVore-Sharpley space \textquotedblleft weak $L^{\infty}"$ defined by
$L(\infty,\infty)=\{f:\left\Vert f\right\Vert _{L(\infty,\infty)}=\sup
_{t>0}\left(  f^{\ast\ast}(t)-f^{\ast}(t)\right)  <\infty\}.$ In fact, this
space is exactly the rearrangement invariant hull of $BMO$ (cf. \cite{bds}).
For a recent account of this part of the story we refer to \cite{corita}.}.

Garsia-Rodemich \cite{garro} introduced a different scale of conditions. Let
$1<p\leq\infty$, then we define,
\[
GaRo_{p}:=GaRo_{p}(Q_{0})=\{f:\left\Vert f\right\Vert _{GaRo_{p}}<\infty\},
\]
where%
\begin{equation}
\left\Vert f\right\Vert _{GaRo_{p}}=\sup_{\{Q_{i}\}\in P}\frac{%
%TCIMACRO{\dsum \limits_{i\in I}}%
%BeginExpansion
{\displaystyle\sum\limits_{i\in I}}
%EndExpansion
\frac{1}{\left\vert Q_{i}\right\vert }\int_{Q_{i}}\int_{Q_{i}}\left\vert
f(x)-f(y)\right\vert dxdy}{\left(
%TCIMACRO{\dsum \limits_{i\in I}}%
%BeginExpansion
{\displaystyle\sum\limits_{i\in I}}
%EndExpansion
\left\vert Q_{i}\right\vert \right)  ^{1/p^{\prime}}}. \label{labuena}%
\end{equation}
The main result about the $GaRo_{p}$ spaces is given by the following (cf.
\cite{garro} for the one dimensional case, \cite{milbmo} for the
$n-$dimensional case): As sets,%
\begin{equation}
GaRo_{p}=\left\{
\begin{array}
[c]{cc}%
L(p,\infty), & 1<p<\infty\\
BMO & \text{if \ }p=\infty
\end{array}
\right.  . \label{delaintro2}%
\end{equation}
In fact, the underlying inequalities can be quantified (cf. \cite{milbmo},
\cite{mil1}). Let $1<p\,<\infty,$ then,%
\begin{equation}
\left\Vert f\right\Vert _{GaRo_{p}}\leq\frac{2p}{p-1}\left\Vert f-f_{Q_{0}%
}\right\Vert _{L(p,\infty)}, \label{belaprima}%
\end{equation}%
\begin{equation}
\left\Vert f-f_{Q_{0}}\right\Vert _{L(p,\infty)}\leq c(n,p)\left\Vert
f\right\Vert _{GaRo_{p}}. \label{bela}%
\end{equation}
Likewise, for $p=\infty,$ we have (cf. \cite{mil1})%
\begin{equation}
\left\Vert f\right\Vert _{GaRo_{\infty}}\simeq\left\Vert f\right\Vert _{BMO},
\label{belas}%
\end{equation}
and (cf. \cite{bds})%
\begin{equation}
\left\Vert f-f_{Q_{0}}\right\Vert _{L(\infty,\infty)}\leq c(n)\left\Vert
f\right\Vert _{BMO}. \label{belast}%
\end{equation}
It is easy to see that\footnote{Indeed, for any $\{Q_{i}\}_{i\in I}\in P$ we
have
\begin{align*}%
%TCIMACRO{\dsum \limits_{i\in I}}%
%BeginExpansion
{\displaystyle\sum\limits_{i\in I}}
%EndExpansion
\frac{1}{\left\vert Q_{i}\right\vert }\int_{Q_{i}}\int_{Q_{i}}\left\vert
f(x)-f(y)\right\vert dxdy  &  \leq2%
%TCIMACRO{\dsum \limits_{i\in I}}%
%BeginExpansion
{\displaystyle\sum\limits_{i\in I}}
%EndExpansion
\int_{Q_{i}}\left\vert f(x)-f_{Q_{i}}\right\vert dx\\
&  =2%
%TCIMACRO{\dsum \limits_{i\in I}}%
%BeginExpansion
{\displaystyle\sum\limits_{i\in I}}
%EndExpansion
\left\vert Q_{i}\right\vert ^{1/p^{\prime}}\left(  \left\vert Q_{i}\right\vert
^{1/p}\frac{1}{\left\vert Q_{i}\right\vert }\int_{Q_{i}}\left\vert
f(x)-f_{Q_{i}}\right\vert dx\right) \\
&  \leq2\left(
%TCIMACRO{\dsum \limits_{i\in I}}%
%BeginExpansion
{\displaystyle\sum\limits_{i\in I}}
%EndExpansion
\left\vert Q_{i}\right\vert \right)  ^{1/p^{\prime}}\left\{
%TCIMACRO{\dsum \limits_{i\in I}}%
%BeginExpansion
{\displaystyle\sum\limits_{i\in I}}
%EndExpansion
\left\vert Q_{i}\right\vert \left(  \frac{1}{\left\vert Q_{i}\right\vert }%
\int_{Q_{i}}\left\vert f(x)-f_{Q_{i}}\right\vert dx\right)  ^{p}\right\}
^{1/p}.
\end{align*}
and (\ref{delaintro}) follows.}%
\begin{equation}
\left\Vert f\right\Vert _{GaRo_{p}}\leq2\left\Vert f\right\Vert _{JN_{p}}.
\label{delaintro}%
\end{equation}
Therefore, combining (\ref{bela}) and (\ref{delaintro}) gives a new proof the
John-Nirenberg embedding (\ref{delaintro3}), and combining (\ref{belas}) with
(\ref{belast}) gives the John-Nirenberg Lemma. Moreover, the Garsia-Rodemich
spaces are particularly well suited to study other important inequalities in
analysis, including Poincar\'{e}-Sobolev embeddings (cf. \cite{mil1}, and
Section \ref{secc::fractional} below), and the basic construction can be
extended to more general settings, e.g. metric spaces, doubling measures, etc.

In this paper we extend the Garsia-Rodemich construction and the scope of its
applications. We first show that the new space, \textbf{B,} introduced by
Bourgain-Brezis-Mironescu \cite{bbm} is closely connected to suitable scalings
of the Garsia-Rodemich norms. As an application we give a new streamlined
approach to the remarkable embedding obtained in \cite{bbm} (cf.
(\ref{ladebmo}) and Theorem \ref{ladetipodebil} \ below),
\begin{equation}
BMO\subset\mathbf{B\ }\subset L(n^{\prime},\infty). \label{delaintro4}%
\end{equation}

The description of the weak $L^{p}$ spaces (cf. (\ref{delaintro2})) via the
Garsia-Rodemich conditions raises a natural question: can one also describe
the $L^{p}$ spaces or other function spaces through Garsia-Rodemich
oscillation conditions? We show that this is indeed the case by means of
modifying a construction of certain martingale spaces, apparently first
introduced by Garsia \cite[$K_{p}^{+}$ spaces, page 165]{garsia}. Let
$X:=X(Q_{0})$ be a rearrangement invariant space (cf. Section
\ref{secc::general} below for background information).

\begin{definition}
\label{def::duca}We shall say that an integrable function $f$ belongs to
$GaRo_{X}$ if there exists $\gamma\in X$ such that for all $\{Q_{i}\}_{i\in
I}\in P$ it holds%
\begin{equation}%
%TCIMACRO{\dsum \limits_{i\in I}}%
%BeginExpansion
{\displaystyle\sum\limits_{i\in I}}
%EndExpansion
\frac{1}{\left\vert Q_{i}\right\vert }\int_{Q_{i}}\int_{Q_{i}}\left\vert
f(x)-f(y)\right\vert dxdy\leq%
%TCIMACRO{\dsum \limits_{i\in I}}%
%BeginExpansion
{\displaystyle\sum\limits_{i\in I}}
%EndExpansion
\int_{Q_{i}}\gamma(x)dx. \label{dulfa}%
\end{equation}
Let
\[
\Gamma_{f}^{X}=\{\gamma\in X:(\ref{dulfa})\text{ holds for all }%
\{Q_{i}\}_{i\in I}\in P\},
\]
and define
\[
\left\Vert f\right\Vert _{GaRo_{X}}=\inf\{\left\Vert \gamma\right\Vert
_{X}:\gamma\in\Gamma_{f}^{X}\}.
\]

\end{definition}

In Section \ref{secc::general} (cf. Theorem \ref{teogrande}) we show that for
rearrangement invariant spaces $X$ whose Boyd indices lie on $(0,1)$%
\footnote{In particular, this class of spaces includes the $L^{p}$ spaces and
the Marcinkiewicz spaces $L(p,\infty),$ $1<p<\infty$.}, we have%

\begin{equation}
GaRo_{X}=X. \label{bruni}%
\end{equation}
In particular, combining (\ref{bruni}) with (\ref{delaintro2}) it follows that%
\[
GaRo_{L(p,\infty)}=L(p,\infty)=GaRo_{p},1<p<\infty.
\]
The proof is based on a suitable extension of the rearrangement inequalities
of \cite{milbmo} (cf. Section \ref{secc::general}, Theorem \ref{teorrea},
below). There exists a constant $c=c(n)$ such that for all $f\in GaRo_{X}$ and
all $\gamma\in\Gamma_{f}^{X},$
\[
f^{\ast\ast}(t)-f^{\ast}(t)\leq c\gamma^{\ast\ast}(t),\text{ }t\in(0,\frac
{1}{4}).
\]

As a consequence, we can extend (\ref{belaprima}) and (\ref{bela}) to the
context of rearrangement invariant spaces (cf. (\ref{lola1}), (\ref{lola2}) below).

In Section \ref{secc::fractional} we use the Garsia-Rodemich conditions to
extend the connection between the space \textbf{B} and fractional Sobolev
embeddings obtained in \cite{bbm}. Let
\begin{equation}
W^{\alpha,p}=\{f:\left\Vert f\right\Vert _{W^{\alpha,p}}<\infty\},
\label{frac}%
\end{equation}
where
\begin{equation}
\left\Vert f\right\Vert _{W^{\alpha,p}}=\left\{  \int_{Q_{0}}\int_{Q_{0}}%
\frac{\left\vert f(x)-f(y)\right\vert ^{p}}{\left\vert x-y\right\vert
^{n+\alpha p}}dxdy\right\}  ^{1/p}. \label{normfrac}%
\end{equation}
Then\footnote{Commenting on an earlier version of this paper, where we had
assumed $p>1,$ Daniel Spector observed that our method of proof also yielded
the case $p=1$ (cf. Remark \ref{daniel} below).} (cf. Theorem \ref{alvarado}
and Remark \ref{daniel} below),
\begin{equation}
W^{\alpha,p}\subset GaRo_{q},1\leq p\leq\frac{n}{\alpha},\frac{1}{q}=\frac
{1}{p}-\frac{\alpha}{n}. \label{srbenson}%
\end{equation}
The generalized Garsia-Rodemich construction can be used to give a far
reaching extension of (\ref{srbenson}) to the setting of rearrangement
invariant spaces. To implement this program we introduce Gagliardo seminorms
adapted to rearrangement invariant spaces as follows. Let $\alpha
\in(0,1),1<p<\infty,$ we formally define
\begin{equation}
D_{p,\alpha}(f)(y)=\left\{  \int_{Q_{0}}\frac{\left\vert f(x)-f(y)\right\vert
^{p}}{\left\vert x-y\right\vert ^{n+\alpha p}}dx\right\}  ^{1/p},y\in Q_{0}.
\label{deformad}%
\end{equation}
Given a rearrangement invariant space $Y:=Y(Q_{0})$ we consider the spaces
defined by%
\begin{equation}
W_{p,Y}^{\alpha}:=W_{p,Y}^{\alpha}(Q_{0}):=\{f:\left\Vert f\right\Vert
_{W_{p,Y}^{\alpha}}=\left\Vert D_{p,\alpha}(f)\right\Vert _{Y}<\infty\}.
\label{deformada}%
\end{equation}

For example, if $Y=L^{p},$ then
\begin{equation}
W_{p,L^{p}}^{\alpha}=W^{\alpha,p}.\label{detodas}%
\end{equation}
Let $X:=X(Q_{0})$ and $Y:=Y(Q_{0})$ be rearrangement invariant spaces such
that the local Riesz potential operator
\[
I_{\alpha,Q_{0}}(f)(y)=\int_{Q_{0}}\frac{f(x)}{\left\vert x-y\right\vert
^{n-a}}dx,y\in Q_{0},
\]
defines a bounded map, $I_{\alpha,Q_{0}}:Y\rightarrow X.$ Then, the following
continuous embedding holds (cf. Theorem \ref{alvarado2} below)%
\[
W_{p,Y}^{\alpha}\subset GaRo_{X}.
\]
It follows that if the Boyd indices of $X$ lie in the interval $(0,1)$ then
(cf. Theorem \ref{teogrande})%
\[
W_{p,Y}^{\alpha}\subset X.
\]
The proof is achieved by means of showing that there exists an absolute
constant $c=c(n,\left\Vert I_{\alpha,Q_{0}}\right\Vert _{Y\rightarrow X})>0$
such that $cI_{\alpha,Q_{0}}(D_{p,\alpha}(f))\in\Gamma_{f}^{X}$ for all $f\in
W_{p,Y}^{\alpha}.$ For example, suppose that $1<p<\frac{n}{\alpha},$ $\frac
{1}{q}=\frac{1}{p}-\frac{\alpha}{n},$ $1\leq r_{1}\leq r_{2}\leq\infty,$ then,
since it is easy to relate mapping properties of $I_{\alpha,Q_{0}}$ to those
of the usual Riesz potential $I_{\alpha}$, and, as is well known for Lorentz
spaces we have (cf. \cite{oneil}), $I_{\alpha}:L(p,r_{1})\rightarrow
L(q,r_{2}),$ we can conclude that\footnote{In particular, if we let
$r_{1}=p<r_{2}=\infty,$ we recover (\ref{srbenson}).} (cf. Example \ref{tordo}
below)
\[
W_{p,L(p,r_{1})}^{\alpha}\subset GaRo_{L(q,r_{2})}=L(q,r_{2}).
\]
The end point inequalities for local Riesz potentials that were obtained in
\cite{brew} can be also implemented here. For example, when $p=\frac{n}%
{\alpha},$ we have (cf. \cite[Theorem 2]{brew}), $I_{\alpha,Q_{0}}:L(\frac
{n}{\alpha},\frac{n}{\alpha})\rightarrow BW_{n/\alpha},$ where%
\[
BW_{n/\alpha}=\{f:\int_{0}^{1}\left(  \frac{f^{\ast}(t)}{(1+\log\frac{1}{t}%
)}\right)  ^{n/\alpha}\frac{dt}{t}<\infty\}.
\]
As a consequence, we obtain the following fractional version of the well known
Brezis-Wainger inequality \cite{brew} (cf. Example \ref{bonafide} below),%
\[
W_{\frac{n}{\alpha}}^{\alpha}=W_{\frac{n}{\alpha},L(\frac{n}{\alpha},\frac
{n}{\alpha})}^{\alpha}\subset GaRo_{BW_{n/\alpha}}.
\]

\textbf{Acknowledgement:} I am grateful to Sergey Astashkin who pointed out
that the original proof of Theorem \ref{ladetipodebil} was not complete. I am
also grateful to Daniel Spector for his proof of Remark \ref{daniel}. I am
also grateful to the anonymous referee who insisted that I should prove
stronger results. Needless to say that I remain responsible for the remaining shortcomings.

\section{The Bourgain-Brezis-Mironescu space and the scaling of Garsia
Rodemich conditions}

Let us observe that neither the norms nor the $GaRo_{p}$ spaces change if we
replace in the definition (\ref{labuena}) the test space $P$ by
\[
\tilde{P}=\{\{Q_{i}\}_{i\in I}:\{Q_{i}\}_{i\in I}\text{ subcubes of }%
Q_{0}\text{ with pairwise disjoint interiors with }\#I<\infty\},
\]
where%
\[
\#I=\text{cardinality of }I.
\]
We now turn to the connection with the Bourgain-Brezis-Mironescu construction.
Given $\varepsilon\in(0,1),$ we let%
\[
\tilde{P}_{\varepsilon}=\{\{Q_{i}\}_{i\in I}:\{Q_{i}\}_{i\in I}\in\tilde
{P},\text{ with side of }Q_{i}=\text{ }\varepsilon\text{ \ for all }i\in
I,\text{ and }\#I\leq\varepsilon^{1-n}\}.
\]
The space $B$ of Bourgain-Brezis-Mironescu is defined by%
\[
B=\{f\in L^{1}(Q_{0}):\left\Vert f\right\Vert _{B}<\infty\},
\]
where
\begin{align*}
\left\Vert f\right\Vert _{B}  &  =\sup_{0<\varepsilon<1}\varepsilon^{n-1}%
\sup_{\{Q_{i}\}\in\tilde{P}_{\varepsilon}}%
%TCIMACRO{\dsum \limits_{i\in I}}%
%BeginExpansion
{\displaystyle\sum\limits_{i\in I}}
%EndExpansion
\frac{1}{\left\vert Q_{i}\right\vert }\frac{1}{\left\vert Q_{i}\right\vert
}\int_{Q_{i}}\int_{Q_{i}}\left\vert f(x)-f(y)\right\vert dxdy\\
&  =\sup_{0<\varepsilon<1}\varepsilon^{-1}\sup_{\{Q_{i}\}\in\tilde
{P}_{\varepsilon}}%
%TCIMACRO{\dsum \limits_{i\in I}}%
%BeginExpansion
{\displaystyle\sum\limits_{i\in I}}
%EndExpansion
\frac{1}{\left\vert Q_{i}\right\vert }\int_{Q_{i}}\int_{Q_{i}}\left\vert
f(x)-f(y)\right\vert dxdy.
\end{align*}

More generally, one can consider the scale of spaces $B_{p},1<p\leq\infty,$
defined by
\[
B_{p}=\{f:\left\Vert f\right\Vert _{B_{p}}=\sup_{0<\varepsilon<1}%
\varepsilon^{-1/p^{\prime}}\sup_{\{Q_{i}\}\in\tilde{P}_{\varepsilon}}%
%TCIMACRO{\dsum \limits_{i\in I}}%
%BeginExpansion
{\displaystyle\sum\limits_{i\in I}}
%EndExpansion
\frac{1}{\left\vert Q_{i}\right\vert }\int_{Q_{i}}\int_{Q_{i}}\left\vert
f(x)-f(y)\right\vert dxdy<\infty\}.
\]
Then, of course the space $B$ corresponds to the choice $p=\infty,$%
\[
B=B_{\infty}.
\]
Now it is easy to see the relationship between $B_{p}$ and $GaRo_{p}.$ First
note that $\tilde{P}_{\varepsilon}\subset\tilde{P},$ moreover, if
$\{Q_{i}\}_{i\in I}\in\tilde{P}_{\varepsilon}$ we have
\[
\left(
%TCIMACRO{\dsum \limits_{i\in I}}%
%BeginExpansion
{\displaystyle\sum\limits_{i\in I}}
%EndExpansion
\left\vert Q_{i}\right\vert \right)  ^{1/p^{\prime}}=\left(  \varepsilon
^{n}(\#I)\right)  ^{1/p^{\prime}}\leq\varepsilon^{1/p^{\prime}},\text{
}0<\varepsilon<1.
\]
Consequently, if $f\in GaRo_{p}$ then, for all $0<\varepsilon<1,$ and for all
$\{Q_{i}\}\in\tilde{P}_{\varepsilon},$%
\[%
%TCIMACRO{\dsum \limits_{i\in I}}%
%BeginExpansion
{\displaystyle\sum\limits_{i\in I}}
%EndExpansion
\frac{1}{\left\vert Q_{i}\right\vert }\int_{Q_{i}}\int_{Q_{i}}\left\vert
f(x)-f(y)\right\vert dxdy\leq\varepsilon^{1/p^{\prime}}\left\Vert f\right\Vert
_{GaRo_{p}}.
\]
In other words,%
\begin{equation}
\left\Vert f\right\Vert _{B_{p}}\leq\left\Vert f\right\Vert _{GaRo_{p}%
},1<p\leq\infty. \label{burka}%
\end{equation}
In particular, for $p=\infty,$ it follows from (\ref{burka}) and (\ref{belas})
that (cf. \cite{bbm})%
\begin{equation}
\left\Vert f\right\Vert _{B}\leq\left\Vert f\right\Vert _{GaRo_{\infty}%
}\approx\left\Vert f\right\Vert _{BMO}. \label{ladebmo}%
\end{equation}

With these preliminaries in place we shall now give an easy proof of a more
refined embedding result that was obtained in \cite{bbm} with a different proof.

\begin{theorem}
\label{ladetipodebil}%
\begin{equation}
B\subset L(n^{\prime},\infty)\text{.} \label{ladeboca}%
\end{equation}

\end{theorem}

\begin{proof}
We will actually show that if $f\in B,$%
\begin{equation}
\left\Vert f\right\Vert _{GaRo_{n^{\prime}}}\preceq\left\Vert f\right\Vert
_{B}. \label{primero}%
\end{equation}
The desired result will then follow from (\ref{bela}) above. We shall show
below that when testing the $GaRo_{n^{\prime}}$ norm it will be enough to
consider dyadic cubes. Let $\mathcal{Q=}\{Q_{i}\}_{i\in I}$ be an arbitrary
element of $P$ formed with dyadic cubes$.$ Following \cite{bbm} we split
$\mathcal{Q}$ as follows. For each $j\in N,$ we consider $\mathcal{F}%
_{j}=\{Q_{i}\in\mathcal{Q}:$ $\left\vert Q_{i}\right\vert =2^{-jn}\},$ then
$\mathcal{Q}=%
%TCIMACRO{\dbigcup \limits_{j=1}^{\infty}}%
%BeginExpansion
{\displaystyle\bigcup\limits_{j=1}^{\infty}}
%EndExpansion
\mathcal{F}_{j}.$ For each $j$ we consider subsets $Q_{\ast}\subset
\mathcal{F}_{j}$ such that%
\[
\#Q_{\ast}\leq\left(  2^{-j}\right)  ^{-(n-1)}=2^{j(n-1)}.
\]
For any such subfamily of cubes $Q_{\ast}$ we have%
\[%
%TCIMACRO{\dsum \limits_{Q_{i}\in Q_{\ast}}}%
%BeginExpansion
{\displaystyle\sum\limits_{Q_{i}\in Q_{\ast}}}
%EndExpansion
\frac{1}{\left\vert Q_{i}\right\vert }\int_{Q_{i}}\int_{Q_{i}}\left\vert
f(x)-f(y)\right\vert dxdy\leq2^{-j}\left\Vert f\right\Vert _{B}.
\]
Therefore, covering $\mathcal{F}_{j}$ with disjoint families of subcubes
$Q_{\ast}$ as above, we find, following the proof in \cite{bbm}, that%
\[%
%TCIMACRO{\dsum \limits_{Q_{i}\in\mathcal{F}_{j}}}%
%BeginExpansion
{\displaystyle\sum\limits_{Q_{i}\in\mathcal{F}_{j}}}
%EndExpansion
\frac{1}{\left\vert Q_{i}\right\vert }\int_{Q_{i}}\int_{Q_{i}}\left\vert
f(x)-f(y)\right\vert dxdy\leq\left(  2^{-j}+2^{-jn}(\#\mathcal{F}_{j})\right)
\left\Vert f\right\Vert _{B}.
\]
Consequently%
\begin{align}%
%TCIMACRO{\dsum \limits_{Q\in\mathcal{Q}}}%
%BeginExpansion
{\displaystyle\sum\limits_{Q\in\mathcal{Q}}}
%EndExpansion
\frac{1}{\left\vert Q\right\vert }\int_{Q}\int_{Q}\left\vert
f(x)-f(y)\right\vert dxdy  &  \leq\left(
%TCIMACRO{\dsum \limits_{j,\mathcal{F}_{j}\neq\varnothing}}%
%BeginExpansion
{\displaystyle\sum\limits_{j,\mathcal{F}_{j}\neq\varnothing}}
%EndExpansion
2^{-j}+%
%TCIMACRO{\dsum \limits_{j}}%
%BeginExpansion
{\displaystyle\sum\limits_{j}}
%EndExpansion
2^{-jn}(\#\mathcal{F}_{j})\right)  \left\Vert f\right\Vert _{B}\nonumber\\
&  =A+B. \label{volvere}%
\end{align}
To estimate $A$ note that if $\mathcal{F}_{j_{0}}\neq\varnothing$ there is
$Q^{0}\in\mathcal{F}_{j_{0}}$ such that
\[
2^{-j_{0}}=\left\vert Q^{0}\right\vert ^{1/n}\leq\left(
%TCIMACRO{\dsum \limits_{Q\in\mathcal{Q}}}%
%BeginExpansion
{\displaystyle\sum\limits_{Q\in\mathcal{Q}}}
%EndExpansion
\left\vert Q\right\vert \right)  ^{1/n}.
\]
Therefore, if we let $j_{0}$ the first index such that $\mathcal{F}_{j_{0}%
}\neq\varnothing$ we have
\begin{align*}%
%TCIMACRO{\dsum \limits_{j,\mathcal{F}_{j}\neq\varnothing}}%
%BeginExpansion
{\displaystyle\sum\limits_{j,\mathcal{F}_{j}\neq\varnothing}}
%EndExpansion
2^{-j}  &  =%
%TCIMACRO{\dsum \limits_{j\geq j_{0},\mathcal{F}_{j}\neq\varnothing}}%
%BeginExpansion
{\displaystyle\sum\limits_{j\geq j_{0},\mathcal{F}_{j}\neq\varnothing}}
%EndExpansion
2^{-j}\\
&  =2^{-j_{0}}%
%TCIMACRO{\dsum \limits_{j\geq j_{0},\mathcal{F}_{j}\neq\varnothing}}%
%BeginExpansion
{\displaystyle\sum\limits_{j\geq j_{0},\mathcal{F}_{j}\neq\varnothing}}
%EndExpansion
2^{-(j-j_{0})}\\
&  \leq\left(
%TCIMACRO{\dsum \limits_{Q\in\mathcal{Q}}}%
%BeginExpansion
{\displaystyle\sum\limits_{Q\in\mathcal{Q}}}
%EndExpansion
\left\vert Q\right\vert \right)  ^{1/n}.
\end{align*}
Term $B$ is estimated by noting that%
\[
2^{-jn}(\#\mathcal{F}_{j})=%
%TCIMACRO{\dsum \limits_{Q\in\mathcal{F}_{j}}}%
%BeginExpansion
{\displaystyle\sum\limits_{Q\in\mathcal{F}_{j}}}
%EndExpansion
\left\vert Q\right\vert
\]
Thus%
\begin{align*}
B  &  =%
%TCIMACRO{\dsum \limits_{j}}%
%BeginExpansion
{\displaystyle\sum\limits_{j}}
%EndExpansion%
%TCIMACRO{\dsum \limits_{Q\in\mathcal{F}_{j}}}%
%BeginExpansion
{\displaystyle\sum\limits_{Q\in\mathcal{F}_{j}}}
%EndExpansion
\left\vert Q\right\vert =%
%TCIMACRO{\dsum \limits_{Q\in\mathcal{Q}}}%
%BeginExpansion
{\displaystyle\sum\limits_{Q\in\mathcal{Q}}}
%EndExpansion
\left\vert Q\right\vert =\left(
%TCIMACRO{\dsum \limits_{Q\in\mathcal{Q}}}%
%BeginExpansion
{\displaystyle\sum\limits_{Q\in\mathcal{Q}}}
%EndExpansion
\left\vert Q\right\vert \right)  ^{1/n}\left(
%TCIMACRO{\dsum \limits_{Q\in\mathcal{Q}}}%
%BeginExpansion
{\displaystyle\sum\limits_{Q\in\mathcal{Q}}}
%EndExpansion
\left\vert Q\right\vert \right)  ^{1/n^{\prime}}\\
&  \leq\left(
%TCIMACRO{\dsum \limits_{Q\in\mathcal{Q}}}%
%BeginExpansion
{\displaystyle\sum\limits_{Q\in\mathcal{Q}}}
%EndExpansion
\left\vert Q\right\vert \right)  ^{1/n}.
\end{align*}
Inserting the estimates of $A$ and $B$ in (\ref{volvere}) we find%
\[%
%TCIMACRO{\dsum \limits_{Q\in\mathcal{Q}}}%
%BeginExpansion
{\displaystyle\sum\limits_{Q\in\mathcal{Q}}}
%EndExpansion
\frac{1}{\left\vert Q\right\vert }\int_{Q}\int_{Q}\left\vert
f(x)-f(y)\right\vert dxdy\leq C\left(
%TCIMACRO{\dsum \limits_{Q\in\mathcal{Q}}}%
%BeginExpansion
{\displaystyle\sum\limits_{Q\in\mathcal{Q}}}
%EndExpansion
\left\vert Q\right\vert \right)  ^{1/n}\left\Vert f\right\Vert _{B}.
\]
It follows that%
\[
\left\Vert f\right\Vert _{GaRo_{n^{\prime}}}\leq C\left\Vert f\right\Vert
_{B}.
\]
To conclude the proof we argue that only dyadic cubes need to be tested.
Indeed, we may assume without loss that $\int_{Q_{0}}f=0.$ Then $\left\Vert
f\right\Vert _{GaRo_{n^{\prime}}}\sim\left\Vert f\right\Vert _{L(n^{\prime
},\infty)}.$ As is well known $L(n^{\prime},\infty)$ can be obtained by the
real method of interpolation (cf. \cite{BS})%
\[
L(n^{\prime},\infty)=(L^{1},L^{\infty})_{1/n,\infty},
\]
with%
\[
\left\Vert f\right\Vert _{L(n^{\prime},\infty)}\sim\sup_{t>0}t^{-1/n}%
[\inf\{\left\Vert h\right\Vert _{L^{1}}+t\left\Vert g\right\Vert _{L^{\infty}%
}:f=h+g\}].
\]
The computation of the indicated infimum (called the $K-$functional of $f$)
can be achieved by elementary cutoffs but can be also achieved using
Calder\'{o}n-Zygmund decompositions. Indeed, a Calder\'{o}n-Zygmund
decomposition $f=h_{CZ}(t)+g_{CZ}(t),$ nearly achieves the infimum and indeed
(cf. \cite{bbm})
\[
\left\Vert f\right\Vert _{L(n^{\prime},\infty)}\sim\sup\{t^{-1/n}\left\Vert
h_{CZ}(t)\right\Vert _{L^{1}}\}.
\]
But as it turns out the computation of $t^{-1/n}\left\Vert h_{CZ}%
(t)\right\Vert _{L^{1}}$ corresponds to the computation of $%
%TCIMACRO{\dsum \limits_{Q_{i}\in\mathcal{Q}}}%
%BeginExpansion
{\displaystyle\sum\limits_{Q_{i}\in\mathcal{Q}}}
%EndExpansion
\frac{1}{\left\vert Q_{i}\right\vert }\int_{Q_{i}}\int_{Q_{i}}\left\vert
f(x)-f(y)\right\vert dxdy$ (cf. \cite{bbm})! Our assertion then follows since
the Calder\'{o}n-Zygmund decomposition used is dyadic.
\end{proof}

\begin{remark}
As we hope it is clear from the proof, the key of the argument is the
introduction of the Garsia-Rodemich conditions.
\end{remark}

\section{A characterization of rearrangement invariant spaces via
Garsia-Rodemich spaces\label{secc::general}}

We start by recalling a few basic notions on rearrangements and rearrangement
invariant spaces. We refer to \cite{BS} and \cite{boyd} for further details
and background.

Let $f:Q_{0}\rightarrow\mathbb{R},$ be a measurable function. The distribution
function of $f$ is given by\footnote{where $\left\vert {}\right\vert $ denotes
the Lebesgue measure on $Q_{0}$ (we also use this notation for the Lebesgue
measure on the unit interval $[0,1]$ since it will cause no confusion).}
\[
\lambda_{f}(t)=\left\vert \{x\in Q_{0}:\left\vert u(x)\right\vert
>t\}\right\vert \text{ \ \ \ \ }(t>0).
\]
The decreasing rearrangement of $f$ is the right-continuous non-increasing
function from $[0,1)$ into $\mathbb{R}^{+}$ which is equimeasurable with $f.$
It can be defined by the formula
\[
f^{\ast}(s)=\inf\{t\geq0:\lambda_{f}(t)\leq s\},\text{ \ }s\in\lbrack0,1),
\]
and satisfies
\[
\lambda_{f}(t)=\left\vert \{x\in Q_{0}:\left\vert f(x)\right\vert
>t\right\vert \}=\left\vert \left\{  s\in\lbrack0,1):f^{\ast}(s)>t\right\}
\right\vert ,\ t\geq0.
\]

The maximal average $f^{\ast\ast}(t)$ is defined by%
\begin{equation}
f^{\ast\ast}(t)=\frac{1}{t}\int_{0}^{t}f^{\ast}(s)ds=\frac{1}{t}\sup\left\{
\int_{E}\left\vert f(x)\right\vert dx:\left\vert E\right\vert =t\right\}
,\text{ }t>0. \label{davoi}%
\end{equation}

We say that a Banach function space $X:=X({Q}_{0}),$ is a
rearrangement-invariant (r.i.) space if $g\in X$ implies that all measurable
functions $f$ with the same rearrangement with $f^{\ast}=g^{\ast},$ also
belong to $X,$ and, moreover, $\Vert f\Vert_{X}=\Vert g\Vert_{X}$.
Rearrangement invariant spaces on ${Q}_{0}$ can be represented by a r.i. space
on the interval $(0,1),$ with Lebesgue measure, $\hat{X}=\hat{X}(0,1),$ such
that\footnote{We refer to \cite[Theorem 4.10 and subsequent remarks]{BS} for
further background information on r.i. spaces.}
\[
\Vert f\Vert_{X}=\Vert f^{\ast}\Vert_{\hat{X}},
\]
for every $f\in X.$ Since it will be clear from the context which space is
involved in the discussion we will \textquotedblleft drop the hat" and denote
the norm of both spaces with same symbol \textquotedblleft$\Vert\circ\Vert
_{X}".$ Typical examples of r.i. spaces are the $L^{p}$-spaces, $L(p,q)$
spaces, Lorentz spaces, Marcinkiewicz spaces and Orlicz spaces.

The following restrictions on the spaces will play a role in our development
in this paper (cf. \cite{boyd}, \cite{BS}):%
\[
(A)\text{ \ \ There exists a universal constant }\beta(X)\text{ such
that\ \ }\Vert f^{\ast\ast}\Vert_{X}\leq\beta_{0}(X)\Vert f\Vert_{X},
\]%
\[
(B)\text{ \ \ \ \ \ \ There exists a universal constant }\beta(X)\text{ such
that\ \ }\Vert\int_{t}^{1}f(s)\frac{ds}{s}\Vert_{X}\leq\beta_{1}(X)\Vert
f\Vert_{X}.
\]
In the language of \textquotedblleft indices" a space that satisfies both
$(A)$ and $(B)$ can be described by saying that \textquotedblleft the Boyd
indices (cf. \cite{boyd}, \cite{BS}) of $X$ are in the interval $(0,1)".$

The basic estimate concerning the $GaRo_{X}$ spaces (cf. Definition
\ref{def::duca} above) is given by

\begin{theorem}
\label{teorrea}Let $X:=X(Q_{0})$ be a rearrangement invariant space. Then
there exists a universal constant $c_{n}>0,$ such that if $f\in GaRo_{X}$ and
$\gamma\in\Gamma_{f}^{X},$ then%
\begin{equation}
f^{\ast\ast}(t)-f^{\ast}(t)\leq c_{n}\gamma^{\ast\ast}(t),0<t<1/4.
\label{vale1}%
\end{equation}

\end{theorem}

\begin{proof}
We follow the proof of \cite[Theorem 5 pages 496-497]{milbmo} very closely and
only indicate in detail the necessary changes at the appropriate steps. Let
$f\in GaRo_{X}$ and $\gamma\in\Gamma_{f}^{X}.$ Since $\left\vert \left\vert
f(x)\right\vert -\left\vert f(y)\right\vert \right\vert \leq\left\vert
f(x)-f(y)\right\vert $, it follows that $\Gamma_{f}^{X}\subset\Gamma
_{\left\vert f\right\vert }^{X}$ and $\left\vert f\right\vert \in GaRo_{X}$.
Moreover, by definition $f^{\ast\ast}(t)=\left\vert f\right\vert ^{\ast\ast
}(t)$ and $f^{\ast}(t)=\left\vert f\right\vert ^{\ast}(t).$ In other words, to
compute the left hand side of (\ref{vale1}) we may assume without loss that
$f$ is positive\footnote{In other words we compute the left hand side using
$\left\vert f\right\vert $ while keeping $\gamma\in\Gamma_{f}^{X}.$}. Fix
$t>0,$ such that $t<\left\vert Q_{0}\right\vert /4=1/4,$ and let $E=\{x\in
Q_{0}:f(x)>f^{\ast}(t)\}.$ By definition, $\left\vert E\right\vert \leq
t<1/4,$ consequently, we can find a relative open subset of $Q_{0},$ $\Omega,$
say$,$ such that $E\subset\Omega$ and $\left\vert \Omega\right\vert \leq
2t\leq1/2.$ By \cite[Lemma 7.2, page 377]{BS} we can find a sequence of cubes
$\{Q_{i}\}_{i\in N},$ with pairwise disjoint interiors, such that:%
\begin{align*}
(i)\text{ \ \ }\left\vert \Omega\cap Q_{i}\right\vert  &  \leq\frac{1}%
{2}\left\vert Q_{i}\right\vert \leq\left\vert \Omega^{c}\cap Q_{i}\right\vert
,\text{ }i=1,2...\\
(ii)\text{ \ \ \ }\Omega &  \subset%
%TCIMACRO{\dbigcup \limits_{i\in N}}%
%BeginExpansion
{\displaystyle\bigcup\limits_{i\in N}}
%EndExpansion
Q_{i}\subset Q_{0}\\
(iii)\text{ \ \ }\left\vert \Omega\right\vert  &  \leq%
%TCIMACRO{\dsum \limits_{i\in N}}%
%BeginExpansion
{\displaystyle\sum\limits_{i\in N}}
%EndExpansion
\left\vert Q_{i}\right\vert \leq2^{n+1}\left\vert \Omega\right\vert .
\end{align*}
At this point following all the corresponding steps in \cite[Theorem 5 pages
496-497]{milbmo} we arrive at
\begin{align*}
t\left(  f^{\ast\ast}(t)-f^{\ast}(t)\right)   &  \leq%
%TCIMACRO{\dsum \limits_{i\in N}}%
%BeginExpansion
{\displaystyle\sum\limits_{i\in N}}
%EndExpansion
\left(  \int_{Q_{i}}\{f(x)-f_{Q_{i}}\}dx+\left\vert E\cap Q_{i}\right\vert
\{f_{Q_{i}}-f^{\ast}(t)\}\right) \\
&  =(I)+(II).
\end{align*}
To estimate $(II),$ we let $J=\{i:f_{Q_{i}}>f^{\ast}(t)\},$ and follow the
steps of \cite[Theorem 5 pages 496-497]{milbmo} until the point we arrive to%
\[
(II)\leq%
%TCIMACRO{\dsum \limits_{i\in J}}%
%BeginExpansion
{\displaystyle\sum\limits_{i\in J}}
%EndExpansion
\frac{1}{\left\vert Q_{i}\right\vert }\int_{Q_{i}}\int_{Q_{i}}\left\vert
f(y)-f(x)\right\vert dxdy.
\]
Now, we recall that $\gamma\in\Gamma_{f}^{X}\subset\Gamma_{\left\vert
f\right\vert }^{X},$ therefore invoking the definition of $GaRo_{X}$ we have%
\begin{align*}
(II)  &  \leq%
%TCIMACRO{\dsum \limits_{i\in J}}%
%BeginExpansion
{\displaystyle\sum\limits_{i\in J}}
%EndExpansion
\int_{Q_{i}}\gamma(x)dx\\
&  \leq\int_{0}^{\sum_{i\in J}\left\vert Q_{i}\right\vert }\gamma^{\ast
}(s)ds\\
&  \leq\int_{0}^{2^{n+2}t}\gamma^{\ast}(s)ds\\
&  =2^{n+2}t\gamma^{\ast\ast}(2^{n+2}t)\\
&  \leq2^{n+2}t\gamma^{\ast\ast}(t)\text{ (since }\gamma^{\ast\ast}\text{ is
decreasing).}%
\end{align*}
Likewise, we estimate $(I)$ proceeding as in \cite[Theorem 5 pages
496-497]{milbmo} until we arrive to
\[
(I)\leq%
%TCIMACRO{\dsum \limits_{i\in N}}%
%BeginExpansion
{\displaystyle\sum\limits_{i\in N}}
%EndExpansion
\frac{1}{\left\vert Q_{i}\right\vert }\int_{Q_{i}}\int_{Q_{i}}\left\vert
f(x)-f(y)\right\vert dxdy.
\]
Again by the definition of $\Gamma_{f}^{X}$ we find%
\[
(I)\leq2^{n+2}t\gamma^{\ast\ast}(t).
\]
Combining the inequalities for $(I)$ and $(II)$ we can find a universal
constant $c_{n}$ such that for $\gamma\in\Gamma_{f}^{X},$ we have%
\[
f^{\ast\ast}(t)-f^{\ast}(t)\leq c_{n}\gamma^{\ast\ast}(t),\text{ for all
}t<1/4,
\]
as we wished to show.
\end{proof}

We can now state and prove the main result of this section.

\begin{theorem}
\label{teogrande}Let $X$ be a rearrangement invariant space with Boyd indices
in the interval $(0,1).$ Then, as sets%
\[
GaRo_{X}=X.
\]
Moreover, we have the following estimates%
\begin{equation}
\left\Vert f-f_{Q_{0}}\right\Vert _{GaRo_{X}}\leq2\left\Vert f\right\Vert _{X}
\label{lola1}%
\end{equation}
and%
\begin{equation}
\left\Vert f-f_{Q_{0}}\right\Vert _{X}\leq c(X)\left\Vert f\right\Vert
_{GaRo_{X}}, \label{lola2}%
\end{equation}
where $c(X)$ depends only on $X.$
\end{theorem}

\begin{proof}
Let us start by remarking that if $f\in X$ then $2\left\vert f\right\vert
\in\Gamma_{f}^{X}.$ Indeed, for any family of cubes $\{Q_{i}\}_{i\in I}\in P$
we have%
\[%
%TCIMACRO{\dsum \limits_{i\in I}}%
%BeginExpansion
{\displaystyle\sum\limits_{i\in I}}
%EndExpansion
\frac{1}{\left\vert Q_{i}\right\vert }\int_{Q_{i}}\int_{Q_{i}}\left\vert
f(x)-f(y)\right\vert dxdy\leq2%
%TCIMACRO{\dsum \limits_{i\in I}}%
%BeginExpansion
{\displaystyle\sum\limits_{i\in I}}
%EndExpansion
\int_{Q_{i}}\left\vert f(x)\right\vert dx.
\]
Therefore, $X\subset GaRo_{X}$ and
\begin{equation}
\left\Vert f\right\Vert _{GaRo_{X}}\leq2\left\Vert f\right\Vert _{X}.
\label{lola}%
\end{equation}
Moreover, since $\left\Vert f\right\Vert _{GaRo_{X}}=\left\Vert f-f_{Q_{0}%
}\right\Vert _{GaRo_{X}},$ we see that (\ref{lola1}) follows from
(\ref{lola}). The remaining inclusion $GaRo_{X}\subset X$ will follow if we
can prove that%
\[
\left\Vert f-f_{Q_{0}}\right\Vert _{X}\leq c(X)\left\Vert f\right\Vert
_{GaRo_{X}}.
\]

Towards this end let $g=f-f_{Q_{0}}.$ Since $g\in L^{1}(Q_{0}),$ we see that
$g^{\ast\ast}(t)\rightarrow0$ as $t\rightarrow\infty.$ Therefore, by the
fundamental theorem of calculus, we can write\footnote{Recall that $\frac
{d}{dt}\left(  g^{\ast\ast}(t)\right)  =\frac{\left(  g^{\ast}(t)-g^{\ast\ast
}(t)\right)  }{t}.$}%
\begin{equation}
g^{\ast\ast}(t)=\int_{t}^{\infty}\left(  g^{\ast\ast}(s)-g^{\ast}(s)\right)
\frac{ds}{s}. \label{vale2}%
\end{equation}
Since $f$ and $g$ differ by a constant we readily see that $\Gamma_{g}%
^{X}=\Gamma_{f}^{X}.$ Consequently, by (\ref{vale1}), for all $\gamma\in$
$\Gamma_{f}^{X}$ we have
\[
\left(  g^{\ast\ast}(t)-g^{\ast}(t)\right)  \leq c_{n}\gamma^{\ast\ast
}(t),t\leq1/4.
\]
To deal with $t>1/4,$ we note that $t\left(  g^{\ast\ast}(t)-g^{\ast
}(t)\right)  =\int_{f^{\ast}(t)}^{\infty}\lambda_{g}(s)ds\leq\int_{0}^{\infty
}\lambda_{g}(s)ds=\left\Vert g\right\Vert _{L^{1}};$ therefore,%
\[
\left(  g^{\ast\ast}(t)-g^{\ast}(t)\right)  \leq t^{-1}\left\Vert g\right\Vert
_{L^{1}},t>1/4.
\]
Inserting the last two estimates in (\ref{vale2}) we find%
\begin{align}
g^{\ast\ast}(t)  &  \leq c_{n}\int_{t}^{1/4}\gamma^{\ast\ast}(s)\frac{ds}%
{s}+c_{n}\left\Vert g\right\Vert _{L^{1}}\int_{1/4}^{\infty}s^{-1}\frac{ds}%
{s}\nonumber\\
&  \leq c_{n}\int_{t}^{\infty}\gamma^{\ast\ast}(s)\frac{ds}{s}+c_{n}%
4\left\Vert g\right\Vert _{L^{1}}. \label{darle}%
\end{align}
Now, writing
\[
\int_{t}^{\infty}\gamma^{\ast\ast}(s)\frac{ds}{s}=\int_{t}^{\infty}\int%
_{0}^{s}\gamma^{\ast}(r)dr\frac{ds}{s^{2}}=\int_{t}^{\infty}\int_{0}^{s}%
\gamma^{\ast}(r)drd(-s^{-1})
\]
we see that%
\[
\int_{t}^{\infty}\gamma^{\ast\ast}(s)\frac{ds}{s}=\gamma^{\ast\ast}%
(t)+\int_{t}^{\infty}\gamma^{\ast}(s)\frac{ds}{s}.
\]
Inserting this information in (\ref{darle}) we find%
\[
g^{\ast\ast}(t)\preceq\gamma^{\ast\ast}(t)+\int_{t}^{\infty}\gamma^{\ast
}(s)\frac{ds}{s}+\left\Vert g\right\Vert _{L^{1}}.
\]
Therefore, applying the $X$ norm on both sides of the last inequality, and
then using the fact that $X$ has Boyd indices lying on $(0,1),$ we
obtain\footnote{Note that $L^{\infty}\subset X,$ which implies that for the
constant function $\left\Vert g\right\Vert _{L^{1}}$ we have
\[
\left\Vert \left\Vert g\right\Vert _{L^{1}}\right\Vert _{X}\preceq\left\Vert
g\right\Vert _{L^{1}}.
\]
}%
\begin{equation}
\left\Vert g\right\Vert _{X}\preceq\left\Vert \gamma\right\Vert _{X}%
+\left\Vert g\right\Vert _{L^{1}}. \label{darle1}%
\end{equation}
Now, since $\int_{Q_{0}}g=0,$ $\left\{  Q_{0}\right\}  \in P,$ and $\left\vert
Q_{0}\right\vert =1,$ we have
\begin{align*}
\left\Vert g\right\Vert _{L^{1}}  &  =\int_{Q_{0}}\left\vert g(x)\right\vert
=\int_{Q_{0}}\left\vert g(x)-\int_{Q_{0}}g\right\vert dx\\
&  \leq\int_{Q_{0}}\int_{Q_{0}}\left\vert g(x)-g(y)\right\vert dxdy\\
&  =\frac{1}{\left\vert Q_{0}\right\vert }\int_{Q_{0}}\int_{Q_{0}}\left\vert
f(x)-f(y)\right\vert dxdy\\
&  \leq\int_{Q_{0}}\left\vert \gamma(y)\right\vert dy\text{ (since }\gamma
\in\Gamma_{f}^{X}\text{)}\\
&  =\left\Vert \gamma\right\Vert _{L^{1}}\\
&  \leq C_{X}\left\Vert \gamma\right\Vert _{X}.
\end{align*}
Updating (\ref{darle1}) we obtain%
\[
\left\Vert g\right\Vert _{X}\preceq\left\Vert \gamma\right\Vert _{X}.
\]
Therefore, taking infimum over all $\gamma\in\Gamma_{f}^{X}$ yields that there
exists an absolute constant $c(X)$, that depends only on $X,$ such that%
\[
\left\Vert f-f_{Q_{0}}\right\Vert _{X}\leq c(X)\left\Vert f\right\Vert
_{GaRo_{X}},
\]
as we wished to show.
\end{proof}

\begin{corollary}
Let $1<p<\infty.$ Then,%
\[
GaRo_{L(p,\infty)}=GaRo_{p}.
\]

\end{corollary}

\begin{proof}
It is well known and easy to see that the $L(p,\infty)$ spaces, $1<p<\infty,$
have Boyd indices in $(0,1)$ (cf. \cite{BS}, \cite{boyd}). Thus, by Theorem
\ref{teogrande}%
\[
GaRo_{L(p,\infty)}=L(p,\infty),
\]
which combined with (\ref{delaintro2}) yields the desired result.
\end{proof}

\begin{remark}
To illustrate the conditions defining $GaRo_{L(p,\infty)}$ and $GaRo_{p},$ we
now give a direct proof of the containment $GaRo_{L(p,\infty)}\subset
GaRo_{p},1<p<\infty.$ We observe that if $\gamma\in\Gamma_{f}^{X}$ then, for
any $\{Q_{i}\}_{i\in I}\in P,$ we have
\begin{align*}%
%TCIMACRO{\dsum \limits_{i\in I}}%
%BeginExpansion
{\displaystyle\sum\limits_{i\in I}}
%EndExpansion
\frac{1}{\left\vert Q_{i}\right\vert }\int_{Q_{i}}\int_{Q_{i}}\left\vert
f(x)-f(y)\right\vert dxdy  &  \leq%
%TCIMACRO{\dsum \limits_{i\in I}}%
%BeginExpansion
{\displaystyle\sum\limits_{i\in I}}
%EndExpansion
\int_{Q_{i}}\left\vert \gamma(x)\right\vert dx\\
&  \leq\int_{Q_{i}}^{\sum_{i\in I}\left\vert Q_{i}\right\vert }\gamma^{\ast
}(s)ds\\
&  \leq\left\Vert \gamma\right\Vert _{L(p,\infty)}\int_{Q_{i}}^{\sum_{i\in
I}\left\vert Q_{i}\right\vert }s^{-1/p}ds\\
&  =p^{\prime}\left\Vert \gamma\right\Vert _{L(p,\infty)}\left(  \sum_{i\in
I}\left\vert Q_{i}\right\vert \right)  ^{1/p^{\prime}}.
\end{align*}
It follows that
\[
\left\Vert f\right\Vert _{GaRo_{p}}\leq p^{\prime}\left\Vert f\right\Vert
_{GaRo_{L(p,\infty)}}.
\]

\end{remark}

\begin{remark}
It is also instructive to compare (\ref{vale1}) with the rearrangement
inequalities in \cite[Theorem 5 (ii)]{milbmo}. For this purpose note that if
$X=L(p,\infty),$ $1<p<\infty,$ then $\gamma\in X$ implies that%
\[
\gamma^{\ast\ast}(t)\leq c_{p}\left\Vert \gamma\right\Vert _{L(p,\infty
)}t^{-1/p},t>0.
\]
Combining the last estimate with (\ref{vale1}), we see that if $f\in
GaRo_{L(p,\infty)},\gamma\in L(p,\infty),$ then%
\[
t^{1/p}\left(  f^{\ast\ast}(t)-f^{\ast}(t)\right)  \leq c_{p}\left\Vert
\gamma\right\Vert _{L(p,\infty)},\text{ for all }t<1/4.
\]
Compare with \cite[Theorem 5 (ii)]{milbmo}.
\end{remark}

\section{Garsia-Rodemich Spaces and Fractional Sobolev
spaces\label{secc::fractional}}

As we have shown in \cite{mil1}, the Garsia-Rodemich formulation of
Marcinkiewicz spaces leads to an easy approach to the self-improvement of
(weak type) Poincar\'{e}-Sobolev inequalities. In this section we discuss
fractional Sobolev inequalities. First we use ideas from \cite{bren},
\cite{bbm}) to prove weak type embeddings\footnote{The method is amenable of
extensions to a much more general context that we shall not pursue here.} of
fractional Sobolev spaces (cf. Subsection \ref{subsecc:weak}). Since strong
type inequalities can be then obtained by the well known method of truncation
of Maz'ya we will not address the issue here. Instead, in Subsection
\ref{subsecc:strong} we take a different approach. Using the generalized
Gagliardo seminorms and generalized Sobolev spaces defined in the Introduction
(cf. (\ref{deformad}) and (\ref{deformada}) above), combined with known
estimates for Riesz potentials, and Theorem \ref{teogrande}, we obtain
embeddings of Fractional Sobolev spaces based on rearrangement invariant spaces.

\subsection{Weak type inequalities\label{subsecc:weak}}

Let $\alpha\in(0,1),$ $1<p<\infty.$ We shall consider the $W^{\alpha,p}$
spaces introduced in the Introduction (cf. (\ref{frac}) and (\ref{normfrac}) above).

\begin{theorem}
\label{alvarado}(i) Let $1<p<\frac{n}{\alpha},$ $\frac{1}{q}=\frac{1}{p}%
-\frac{\alpha}{n}.$ Then,
\begin{equation}
\left\Vert f\right\Vert _{GaRo_{q}}\leq n^{\frac{(n+\alpha p)}{p2}}\left\Vert
f\right\Vert _{W^{\alpha,p}}. \label{limite}%
\end{equation}

(ii) If $p=\frac{n}{\alpha},$ then%
\[
\left\Vert f\right\Vert _{GaRo_{\infty}}\leq n^{\alpha}\left\Vert f\right\Vert
_{W^{\alpha,\frac{n}{\alpha}}}.
\]

\end{theorem}

\begin{proof}
(i) Let $f\in W^{\alpha,p},$ and let $\{Q_{i}\}_{i\in I}$ be an arbitrary
element of $\tilde{P}.$ Let%
\[
A=%
%TCIMACRO{\dsum \limits_{i\in I}}%
%BeginExpansion
{\displaystyle\sum\limits_{i\in I}}
%EndExpansion
\frac{1}{\left\vert Q_{i}\right\vert }\int_{Q_{i}}\int_{Q_{i}}\left\vert
f(x)-f(y)\right\vert dxdy.
\]
Estimating the distance between two points of $Q_{i}$ by the diameter of
$Q_{i}$ we find that for any $r>0,$
\begin{align}
\frac{1}{\left\vert Q_{i}\right\vert }\int_{Q_{i}}\int_{Q_{i}}\left\vert
f(x)-f(y)\right\vert dxdy  &  =\frac{1}{\left\vert Q_{i}\right\vert }%
\int_{Q_{i}}\int_{Q_{i}}\frac{\left\vert f(x)-f(y)\right\vert }{\left\vert
x-y\right\vert ^{r}}\left\vert x-y\right\vert ^{r}dxdy\\
&  \leq\frac{n^{r/2}\left\vert Q_{i}\right\vert ^{r/n}}{\left\vert
Q_{i}\right\vert }\int_{Q_{i}}\int_{Q_{i}}\frac{\left\vert
f(x)-f(y)\right\vert }{\left\vert x-y\right\vert ^{r}}dxdy. \label{sumar}%
\end{align}
Let $r=\frac{(n+\alpha p)}{p}$, then summing (\ref{sumar}) and applying
H\"{o}lder's inequality (twice) we obtain%
\begin{align}
A  &  \leq%
%TCIMACRO{\dsum \limits_{i\in I}}%
%BeginExpansion
{\displaystyle\sum\limits_{i\in I}}
%EndExpansion
\frac{n^{r/2}\left\vert Q_{i}\right\vert ^{r/n}}{\left\vert Q_{i}\right\vert
}\int_{Q_{i}}\int_{Q_{i}}\frac{\left\vert f(x)-f(y)\right\vert }{\left\vert
x-y\right\vert ^{r}}dxdy\nonumber\\
&  \leq n^{r/2}%
%TCIMACRO{\dsum \limits_{i\in I}}%
%BeginExpansion
{\displaystyle\sum\limits_{i\in I}}
%EndExpansion
\left(  \frac{\left\vert Q_{i}\right\vert ^{r/n}}{\left\vert Q_{i}\right\vert
}\left\vert Q_{i}\right\vert ^{2/p^{\prime}}\right)  \left\{  \int_{Q_{i}}%
\int_{Q_{i}}\frac{\left\vert f(x)-f(y)\right\vert ^{p}}{\left\vert
x-y\right\vert ^{n+\alpha p}}dxdy\right\}  ^{1/p}\nonumber\\
&  \leq n^{r/2}\left\{
%TCIMACRO{\dsum \limits_{i\in I}}%
%BeginExpansion
{\displaystyle\sum\limits_{i\in I}}
%EndExpansion
\left\{  \frac{\left\vert Q_{i}\right\vert ^{r/n}}{\left\vert Q_{i}\right\vert
}\left\vert Q_{i}\right\vert ^{2/p^{\prime}}\right\}  ^{p^{\prime}}\right\}
^{1/p^{\prime}}\left\{
%TCIMACRO{\dsum \limits_{i\in I}}%
%BeginExpansion
{\displaystyle\sum\limits_{i\in I}}
%EndExpansion
\int_{Q_{i}}\int_{Q_{i}}\frac{\left\vert f(x)-f(y)\right\vert ^{p}}{\left\vert
x-y\right\vert ^{n+\alpha p}}dxdy\right\}  ^{1/p}\nonumber\\
&  \leq n^{r/2}\left\{
%TCIMACRO{\dsum \limits_{i\in I}}%
%BeginExpansion
{\displaystyle\sum\limits_{i\in I}}
%EndExpansion
\left\{  \frac{\left\vert Q_{i}\right\vert ^{r/n}}{\left\vert Q_{i}\right\vert
}\left\vert Q_{i}\right\vert ^{2/p^{\prime}}\right\}  ^{p^{\prime}}\right\}
^{1/p^{\prime}}\left\Vert f\right\Vert _{W^{\alpha,p}} \label{laotra}%
\end{align}
Now, by computation $\left(  \frac{r}{n}-1\right)  p^{\prime}+2=\frac{\alpha
}{n}p^{\prime}+1,$ and therefore%
\begin{align*}%
%TCIMACRO{\dsum \limits_{i\in I}}%
%BeginExpansion
{\displaystyle\sum\limits_{i\in I}}
%EndExpansion
\left\{  \frac{\left\vert Q_{i}\right\vert ^{r/n}}{\left\vert Q_{i}\right\vert
}\left\vert Q_{i}\right\vert ^{2/p^{\prime}}\right\}  ^{p^{\prime}}  &  =%
%TCIMACRO{\dsum \limits_{i\in I}}%
%BeginExpansion
{\displaystyle\sum\limits_{i\in I}}
%EndExpansion
\left\vert Q_{i}\right\vert ^{(\frac{\alpha}{n}p^{\prime}+1)}\\
&  \leq\left\{
%TCIMACRO{\dsum \limits_{i\in I}}%
%BeginExpansion
{\displaystyle\sum\limits_{i\in I}}
%EndExpansion
\left\vert Q_{i}\right\vert \right\}  ^{(\frac{\alpha}{n}p^{\prime}+1)}.
\end{align*}
Inserting this information in (\ref{laotra}) yields%
\begin{equation}
A\leq n^{r/2}\left\{
%TCIMACRO{\dsum \limits_{i\in I}}%
%BeginExpansion
{\displaystyle\sum\limits_{i\in I}}
%EndExpansion
\left\vert Q_{i}\right\vert \right\}  ^{(\frac{\alpha}{n}p^{\prime}+1)\frac
{1}{p^{\prime}}}\left\Vert f\right\Vert _{W^{\alpha,p}}. \label{laver}%
\end{equation}
Since $(\frac{\alpha}{n}p^{\prime}+1)\frac{1}{p^{\prime}}=\frac{1}{q^{\prime}%
},$ we obtain%
\[
\left\Vert f\right\Vert _{GaRo_{q}}\leq n^{r/2}\left\Vert f\right\Vert
_{W^{\alpha,p}}.
\]

(ii) In the limiting case, $p=\frac{n}{\alpha},$ therefore $p^{\prime}%
=\frac{n}{n-\alpha}$ and we see that $(\frac{\alpha}{n}p^{\prime}%
+1)=p^{\prime};$ consequently, by definition, $\frac{1}{q^{\prime}}=1.$ Let
$r=\frac{2n}{p},$ inserting this information in (\ref{laver}) yields%
\[
A\leq n^{\alpha}\left\{
%TCIMACRO{\dsum \limits_{i\in I}}%
%BeginExpansion
{\displaystyle\sum\limits_{i\in I}}
%EndExpansion
\left\vert Q_{i}\right\vert \right\}  \left\Vert f\right\Vert _{W^{\alpha
,\frac{n}{\alpha}}.}%
\]
Thus,%
\begin{equation}
\left\Vert f\right\Vert _{GaRo_{\infty}}\leq n^{\alpha}\left\Vert f\right\Vert
_{W^{\alpha,\frac{n}{\alpha}}.} \label{laver1}%
\end{equation}

\end{proof}

\begin{remark}
\label{daniel}By private correspondence Daniel Spector observed that with a
minor modification the proof also works in the case $p=1.$ Indeed, if $p=1,$
we let $r=n+\alpha$, and proceed as in the proof of case (i), but now only one
application of H\"{o}lder's inequality is needed to obtain%
\begin{align*}
A  &  \leq n^{(n+\alpha)/2}\sup_{i\in I}\left\{  \left\vert Q_{i}\right\vert
^{\frac{\alpha}{n}}\right\}  \left\{
%TCIMACRO{\dsum \limits_{i\in I}}%
%BeginExpansion
{\displaystyle\sum\limits_{i\in I}}
%EndExpansion
\int_{Q_{i}}\int_{Q_{i}}\frac{\left\vert f(x)-f(y)\right\vert }{\left\vert
x-y\right\vert ^{n+\alpha}}dxdy\right\} \\
&  \leq n^{(n+\alpha)/2}\left\{
%TCIMACRO{\dsum \limits_{i\in I}}%
%BeginExpansion
{\displaystyle\sum\limits_{i\in I}}
%EndExpansion
\left\vert Q_{i}\right\vert \right\}  ^{\frac{\alpha}{n}}\left\Vert
f\right\Vert _{W^{\alpha,1}.}%
\end{align*}
Since $\frac{1}{q}=1-\frac{\alpha}{n}$ we thus have,%
\begin{equation}
\left\Vert f\right\Vert _{GaRo_{q}}\leq n^{(n+\alpha)/2}\left\Vert
f\right\Vert _{W^{\alpha,1}}. \label{obtenida}%
\end{equation}

Alternatively the inequality (\ref{obtenida}) can be obtained letting
$p\rightarrow1$ in (\ref{limite}).
\end{remark}

\begin{remark}
Note that starting with (\ref{sumar}) applied to a cube $Q$ with $r=\frac
{2n}{p},$ and then applying H\"{o}lder's inequality, yields%
\[
\frac{1}{\left\vert Q\right\vert ^{2}}\int_{Q}\int_{Q}\left\vert
f(x)-f(y)\right\vert dxdy\leq\frac{n^{\alpha}\left\vert Q\right\vert ^{2/p}%
}{\left\vert Q\right\vert ^{2}}\left\vert Q\right\vert ^{2/p^{\prime}}\left(
\int_{Q}\int_{Q}\frac{\left\vert f(x)-f(y)\right\vert ^{p}}{\left\vert
x-y\right\vert ^{2n}}dxdy\right)  ^{1/p},
\]
and (\ref{laver1}) follows from the fact that%
\[
\left\Vert f\right\Vert _{GaRo_{\infty}}=\left\Vert f\right\Vert _{BMO}%
\simeq\sup_{Q\subset Q_{0}}\frac{1}{\left\vert Q\right\vert ^{2}}\int_{Q}%
\int_{Q}\left\vert f(x)-f(y)\right\vert dxdy.
\]
This approach to (\ref{laver1}) is classical (cf. \cite{bren}); the use of the
Garsia-Rodemich spaces unifies the proof of (i) and (ii).
\end{remark}

\subsection{Strong type inequalities\label{subsecc:strong}}

The characterization of rearrangement invariant spaces $X$ with indices lying
on $(0,1)$ as $GaRo_{X}$ spaces provided by Theorem \ref{teogrande} allows to
unify the proofs of the weak and strong type Sobolev inequalities in the
general context of rearrangement invariant spaces. For other treatments of
fractional Sobolev inequalities we refer to \cite{bbm2}, \cite{beckner},
\cite{ponce}, \cite{fran}, \cite{xiao} and the references therein.

Let $\alpha\in(0,1),$ $1\leq p<\infty$. Recall the definition of the main
objects of study:
\[
D_{p,\alpha}(f)(y)=\left\{  \int_{Q_{0}}\frac{\left\vert f(x)-f(y)\right\vert
^{p}}{\left\vert x-y\right\vert ^{n+\alpha p}}dx\right\}  ^{1/p},y\in Q_{0}.
\]%
\[
W_{p,Y}^{\alpha}:=W_{p,Y}^{\alpha}(Q_{0}):=\{f:\left\Vert D_{p,\alpha
}(f)\right\Vert _{Y}<\infty\}.
\]

\[
I_{\alpha,Q_{0}}(f)(x)=\int_{Q_{0}}\frac{f(y)}{\left\vert x-y\right\vert
^{n-\alpha}}dy,x\in Q_{0}.
\]

Our main result reads as follows

\begin{theorem}
\label{alvarado2}Let $\alpha\in(0,1),1<p<\frac{n}{\alpha},$ $\frac{1}{q}%
=\frac{1}{p}-\frac{\alpha}{n}.$ Let $X$ and $Y$ be rearrangement invariant
spaces such that $I_{\alpha,Q_{0}}$ is a bounded map, $I_{\alpha,Q_{0}%
}:Y\rightarrow X.$ Then,%
\[
W_{p,Y}^{\alpha}\subset GaRo_{X}.
\]

\end{theorem}

\begin{proof}
Let $f\in W_{p,Y}^{\alpha},$ and let $\{Q_{i}\}_{i\in I}$ be an arbitrary
element of $\tilde{P}.$ Let%
\[
A=%
%TCIMACRO{\dsum \limits_{i\in I}}%
%BeginExpansion
{\displaystyle\sum\limits_{i\in I}}
%EndExpansion
\frac{1}{\left\vert Q_{i}\right\vert }\int_{Q_{i}}\int_{Q_{i}}\left\vert
f(x)-f(y)\right\vert dxdy.
\]

Using Jensen's inequality on the inner integral we find that%
\begin{align}
A  &  \leq%
%TCIMACRO{\dsum \limits_{i\in I}}%
%BeginExpansion
{\displaystyle\sum\limits_{i\in I}}
%EndExpansion
\int_{Q_{i}}\left(  \frac{1}{\left\vert Q_{i}\right\vert }\int_{Q_{i}%
}\left\vert f(x)-f(y)\right\vert ^{p}dx\right)  ^{1/p}dy\nonumber\\
&  \leq n^{(n+\alpha p)/2p}%
%TCIMACRO{\dsum \limits_{i\in I}}%
%BeginExpansion
{\displaystyle\sum\limits_{i\in I}}
%EndExpansion
\int_{Q_{i}}\left(  \frac{\left\vert Q_{i}\right\vert ^{\frac{n+\alpha p}{n}}%
}{\left\vert Q_{i}\right\vert }\int_{Q_{i}}\frac{\left\vert
f(x)-f(y)\right\vert ^{p}}{\left\vert x-y\right\vert ^{n+\alpha p}}dx\right)
^{1/p}dy\nonumber\\
&  \leq n^{(n+\alpha p)/2p}%
%TCIMACRO{\dsum \limits_{i\in I}}%
%BeginExpansion
{\displaystyle\sum\limits_{i\in I}}
%EndExpansion
\frac{\left\vert Q_{i}\right\vert ^{\frac{n+\alpha p}{np}}}{\left\vert
Q_{i}\right\vert ^{1/p}}\int_{Q_{i}}\left(  \int_{Q_{i}}\frac{\left\vert
f(x)-f(y)\right\vert ^{p}}{\left\vert x-y\right\vert ^{n+\alpha p}}dx\right)
^{1/p}dy\nonumber\\
&  \leq n^{(n+\alpha p)/2p}%
%TCIMACRO{\dsum \limits_{i\in I}}%
%BeginExpansion
{\displaystyle\sum\limits_{i\in I}}
%EndExpansion
\frac{\left\vert Q_{i}\right\vert ^{\frac{n+\alpha p}{np}}}{\left\vert
Q_{i}\right\vert ^{1/p}}\int_{Q_{i}}\left(  \int_{Q_{0}}\frac{\left\vert
f(x)-f(y)\right\vert ^{p}}{\left\vert x-y\right\vert ^{n+\alpha p}}dx\right)
^{1/p}dy\nonumber\\
&  =n^{(n+\alpha p)/2p}%
%TCIMACRO{\dsum \limits_{i\in I}}%
%BeginExpansion
{\displaystyle\sum\limits_{i\in I}}
%EndExpansion
\frac{\left\vert Q_{i}\right\vert ^{\frac{n+\alpha p}{np}}}{\left\vert
Q_{i}\right\vert ^{1/p}}\int_{Q_{i}}D_{p,\alpha}(f)(y)dy\nonumber\\
&  =n^{(n+\alpha p)/2p}%
%TCIMACRO{\dsum \limits_{i\in I}}%
%BeginExpansion
{\displaystyle\sum\limits_{i\in I}}
%EndExpansion
\left\vert Q_{i}\right\vert ^{\frac{n+\alpha p}{np}-\frac{1}{p}}\left\vert
Q_{i}\right\vert ^{-1}\left\vert Q_{i}\right\vert \int_{Q_{i}}D_{p,\alpha
}(f)(y)dy\nonumber\\
&  =n^{(n+\alpha p)/2p}%
%TCIMACRO{\dsum \limits_{i\in I}}%
%BeginExpansion
{\displaystyle\sum\limits_{i\in I}}
%EndExpansion
\left\vert Q_{i}\right\vert ^{\frac{n+\alpha p}{np}-\frac{1}{p}-1}\int_{Q_{i}%
}\int_{Q_{i}}D_{p,\alpha}(f)(y)dydz\nonumber\\
&  =n^{(n+\alpha p)/2p}%
%TCIMACRO{\dsum \limits_{i\in I}}%
%BeginExpansion
{\displaystyle\sum\limits_{i\in I}}
%EndExpansion
\left\vert Q_{i}\right\vert ^{\frac{\alpha}{n}-1}\int_{Q_{i}}\int_{Q_{i}%
}D_{p,\alpha}(f)(y)dydz\nonumber\\
&  \leq n^{(n+\alpha p)/2p}n^{\frac{n-\alpha}{2}}%
%TCIMACRO{\dsum \limits_{i\in I}}%
%BeginExpansion
{\displaystyle\sum\limits_{i\in I}}
%EndExpansion
\int_{Q_{i}}\int_{Q_{i}}\frac{D_{p,\alpha}(f)(y)}{\left\vert y-z\right\vert
^{n-\alpha}}dydz\nonumber\\
&  \leq n^{(n+\alpha p)/2p}n^{\frac{n-\alpha}{2}}%
%TCIMACRO{\dsum \limits_{i\in I}}%
%BeginExpansion
{\displaystyle\sum\limits_{i\in I}}
%EndExpansion
\int_{Q_{i}}\int_{Q_{0}}\frac{D_{p,\alpha}(f)(y)}{\left\vert y-z\right\vert
^{n-\alpha}}dydz\nonumber\\
&  =C_{n}%
%TCIMACRO{\dsum \limits_{i\in I}}%
%BeginExpansion
{\displaystyle\sum\limits_{i\in I}}
%EndExpansion
\int_{Q_{i}}I_{\alpha,Q_{0}}(D_{p,\alpha}(f))(z)dz. \label{tipica}%
\end{align}
Moreover, by assumption,%
\begin{align*}
\left\Vert I_{\alpha,Q_{0}}(D_{p,\alpha}(f))\right\Vert _{X}  &
\leq\left\Vert I_{\alpha,Q_{0}}\right\Vert _{Y\rightarrow X}\left\Vert
D_{p,\alpha}(f)\right\Vert _{Y}\\
&  =\left\Vert I_{\alpha,Q_{0}}\right\Vert _{Y\rightarrow X}\left\Vert
f\right\Vert _{W_{p,Y}^{\alpha}}.
\end{align*}
Combining this fact with (\ref{tipica}) we see that $C_{n}I_{\alpha,Q_{0}%
}(D_{p,\alpha}(f))\in\Gamma_{f}^{X},$ and%
\[
\left\Vert f\right\Vert _{GaRo_{X}}\leq C_{n}\left\Vert I_{\alpha,Q_{0}%
}\right\Vert _{Y\rightarrow X}\left\Vert f\right\Vert _{W_{p,Y}^{\alpha}},
\]
as we wished to show.
\end{proof}

\begin{corollary}
Suppose that all the assumptions of Theorem \ref{alvarado2} hold and,
furthermore, that $X$ has Boyd indices lying in $(0,1).$ Then,%
\[
W_{p,Y}^{\alpha}\subset X.
\]

\end{corollary}

\begin{proof}
Applying successively Theorem \ref{alvarado2} and Theorem \ref{teogrande} we
obtain%
\[
W_{p,Y}^{\alpha}\subset GaRo_{X}=X.
\]

\end{proof}

\begin{example}
\label{tordo}The Lorentz spaces $L(s,r),1<s<\infty,1\leq r\leq\infty,$ are
defined by the condition%
\[
\left\Vert f\right\Vert _{L(s,r)}=\left\{  \int_{0}^{\infty}\left(
f^{\ast\ast}(u)u^{1/s}\right)  ^{r}\frac{du}{u}\right\}  ^{1/r}<\infty.
\]
It is well known (cf. \cite{oneil}) that if $1<p<\frac{n}{\alpha},$ $\frac
{1}{q}=\frac{1}{p}-\frac{\alpha}{n},1\leq r_{1}\leq r_{2}\leq\infty,$ then
$I_{\alpha,Q_{0}}:L(p,r_{1})\rightarrow L(q,r_{2}),$is a bounded map. Then we
can conclude that\footnote{In particular, if we let $r_{1}=p<r_{2}=\infty,$ we
recover (\ref{srbenson}).}
\[
W_{p,L(p,r_{1})}^{\alpha}\subset GaRo_{L(q,r_{2})}=L(q,r_{2}).
\]

\end{example}

\begin{example}
The previous discussion shows that%
\begin{equation}
\left\Vert f\right\Vert _{L^{q}}\leq c(q)\left\Vert f\right\Vert
_{GaRo_{L^{q}}}\leq c(q)C_{n}\left\Vert I_{\alpha,Q_{0}}\right\Vert
_{L^{p}\rightarrow L^{q}}\left\Vert f\right\Vert _{W_{p,L^{p}}^{\alpha}}.
\label{erminia}%
\end{equation}
Note that since%
\[
W_{p,L^{p}}^{\alpha}=W^{\alpha,p},
\]
(\ref{erminia}) gives the corresponding strong type inequalities of Theorem
\ref{alvarado}.
\end{example}

\begin{example}
\label{bonafide}The end point inequalities for local Riesz potentials that
were obtained in \cite{brew} can be also implemented here. For example, when
$p=\frac{n}{\alpha},$ we have (cf. \cite[Theorem 2]{brew}), $I_{\alpha,Q_{0}%
}:L(\frac{n}{\alpha},\frac{n}{\alpha})\rightarrow BW_{n/\alpha},$ where%
\[
BW_{n/\alpha}=\{f:\left\Vert f\right\Vert _{BW_{n/\alpha}}=\left\{  \int%
_{0}^{1}\left(  \frac{f^{\ast}(t)}{(1+\log\frac{1}{t})}\right)  ^{n/\alpha
}\frac{dt}{t}\right\}  ^{\alpha/n}<\infty\}.
\]
As a consequence, we obtain the following fractional version of the well known
Brezis-Wainger inequality \cite{brew},%
\[
W_{\frac{n}{\alpha}}^{\alpha}=W_{\frac{n}{\alpha},L(\frac{n}{\alpha},\frac
{n}{\alpha})}^{\alpha}\subset GaRo_{BW_{n/\alpha}}.
\]

\end{example}

\section{Final Remarks}

Finally we comment briefly on some loose ends that we are leaving for future work.

\begin{itemize}
\item We should mention the interesting work by
Ambrosio-Bourgain-Brezis-Figalli (cf. \cite{bre1}) on isotropic versions of
\textbf{B }and their use in the computation of perimeters of sets. What is the
connection with scalings of Garsia-Rodemich conditions and the limits
fractional Sobolev norms (cf. \cite{fusco}, \cite{ps}, \cite{fraccionario})?

\item It could be interesting to explore the role of Garsia-Rodemich
conditions in solving the problem of proving dimension free versions of the
John-Nirenberg inequality (cf. \cite{cwikel}).

\item As it often happens in mathematics, the generalization of the
Garsia-Rodemich condition developed in this paper makes it easier to
understand the connection with other constructions. In a forthcoming joint
paper with Sergey Astashkin \cite{asmi} we connect the Garsia-Rodemich spaces
to the Fefferman-Stein inequalities and interpolation theory.
\end{itemize}

\end{document}